\documentclass[12pt,a4paper]{amsart}
\usepackage[letterpaper, margin=0.8in]{geometry}
\usepackage{amsfonts,amssymb,amsthm,epsf, graphics, bbm, mathabx, verbatim, caption, subcaption, enumerate}
\usepackage{comment}
\usepackage{hyperref}
\usepackage{diagbox}
\usepackage{amsmath}
\usepackage{amsthm}
\usepackage{amssymb}
\usepackage{cite}
\usepackage{tikz}
\usetikzlibrary{patterns}
\usepackage{float}

\newtheorem{proposition}{Proposition}

\newtheorem{lemma}{Lemma}
\theoremstyle{definition}

\newtheorem{theorem}{Theorem}
\newtheorem*{theoremA}{Theorem A}

\newtheorem{example}{Example}

\newtheorem{problem}{Problem}

\theoremstyle{remark}
\newtheorem{remark}{Remark}

\newcommand{\R}{\mathbb{R}}

\newcommand{\eps}{\varepsilon}

\title{Pinned Distances and Density Theorems in $\R^d$}
\author{Chenjian Wang}
\date{}

\begin{document}

\begin{abstract}
We study a pinned variant of Bourgain’s theorem \cite{Bourgain1986}, concerning the occurrence of affine copies of $k$-point patterns in $\R^d$.  Focusing on the case $k=2$, which corresponds to pinned distances, we show that the classical conclusion does not extend to the pinned setting: there exist sets of positive upper density in $\R^d,~d\geq 2$ such that no single pinned point determines all sufficiently large distances. However, we establish a weaker quantitative result: for every point $x$ in such a set, the pinned distance set at $x$ has (one-dimensional) positive upper density. We also construct an example demonstrating the sharpness of this bound. These findings highlight a structural distinction between global and pinned configurations.

\end{abstract} 

 \maketitle

\let\thefootnote\relax

\section{Introduction}
Let us define the \textit{upper density} of a Lebesgue measurable set $A\subseteq \R^d$ as
\begin{equation*}
    \delta(A):=\limsup_{R\to \infty} \frac{|B(0,R)\bigcap A|}{R^d}
\end{equation*}
where $|\cdot|$ denotes the Lebesgue measure in $\R^d$. For computational reason, we replace by $R^d$ the $|B(0,R)|$ in the denominator in the standard definition, the difference is not essential.

In 1986, Bourgain proved a  Szemer\'edi-type theorem for simplices. We state its planar case.

\begin{theoremA}[Bourgain \cite{Bourgain1986}]\label{Bourgain's theorem}
    Suppose set $A\subseteq \R^2$ has positive upper density.
   Then there exists a number $l>0$ such that $(l,\infty)\subseteq D(A)$, where $D(A)$ is the distance set of $A$, defined as 
   \begin{equation*}
       D(A):=\{|x-y|:x,y\in A\}.
   \end{equation*}

\end{theoremA}


Inspired by the work on pinned distance set (see equation \eqref{pinneddisset} for its definition) problems \cite{LiuFormular,GIOW5/4}, we study a pinned analog of Bourgain’s theorem. Specifically, we consider whether a single point 
$x$ in $A$ can serve as the base for observing all large distances. More precisely,
\begin{problem}\label{problem1}
    Suppose $A\subseteq \R^d, d\geq 2$ has positive upper density. Is there any $x\in A$ and a number $l(x)>0$ such that for all $l'>l(x)$,  there exists $ y\in A \text{ such that } |x-y|=l'$?
\end{problem}

This is a natural way to propose the pinned version of Bourgain's theorem. In the study of the Falconer's distance set problem in geometric measure theory, people asked the following question. 
\begin{problem}[Falconer's distance set problem, pinned version]\label{PinnedFalconer}
    Assume $A\subseteq \R^d,d\geq 2$ is a Borel set and $\dim_ H(A)>\frac{d}{2}$. Then is there any $x\in A$ such that the pinned distance set 
    \begin{equation}\label{pinneddisset}
        D_x(A):=\{|x-y|:y\in A\}
    \end{equation}
    has positive Lebesgue measure? Here $\dim_ H(A)$ is the Hausdorff dimension of $A$.
\end{problem}
Comparing these two problems, in both cases, we have two ``original" sets with certain ``largeness" conditions: In Problem \ref{problem1}, the upper density is positive, whereas in Problem \ref{PinnedFalconer}, it is $\dim_ H(A)>\frac{d}{2}$. The conclusions are both about the ``largeness" of certain ``induced" sets of the original sets, which is pinned distance set in both cases. However, the manners in which we quantify largeness are different. In Problem \ref{problem1}, we use $(l(x),\infty)\subseteq D_x(A)$, whereas in Problem \ref{PinnedFalconer}, we use $|D_x(A)|>0$.

Returning to Problem \ref{problem1}, our first result provides a negative answer to it.
\begin{proposition}\label{counterexample}
     For $d\geq 2$, there exists a set $ E\subseteq \R^d$ with $\delta(E)>0$ such that for all $ x\in E$ and all $ l>0$, there exists $ l'>l$ such that
    \begin{equation*}
        |x-y|\neq l',~ \forall y\in E.
    \end{equation*}
\end{proposition}
The proof is a construction of such set $E$ showing that the pinned analog of Bourgain’s result fails in the strongest form. It will be presented in {Example} \ref{NoBourgainTypeResult}.

Proposition \ref{counterexample} tells us that it is inappropriate to quantify the sizes of the pinned distance sets by ``containing $(l(x),\infty)$". This conclusion turns out to be too strong to be true. Example \ref{NoBourgainTypeResult} motivates the search for weaker guarantees on the size of the pinned distance set. 
Therefore, our second result answers this quantitatively.

\begin{theorem} \label{prop1}
     Assume $d\geq 2$ and $A\subseteq \R^d$ with $\delta(A)>0$. Then for all $x\in \R^d,$
    $$ \delta(D_x(A))=\limsup_{R\to \infty}\frac{|D_x(A)\bigcap (-R,R)|}{R}\geq \frac{\delta(A)}{2|\mathbb S^{d-1}|},$$
    where $D_x(A)$ is  defined in \eqref{pinneddisset} and $|\mathbb S^{d-1}|$ is the area of the $(d-1)$-dimensional unit sphere.
\end{theorem}
The theorem affirms that although the pinned distance set $D_x(A)$ does not contain certain interval $(l(x),\infty)$; its upper density, as a subset of $\R$, must be positive. 

The proof of Theorem \ref{prop1} is elementary. However, the result is sharp in the 
following sense. 
\begin{proposition}[sharpness of the Theorem \ref{prop1}]\label{Prop:Sharpness}
   For $d\geq 2$, there are a constant $C_d$ and a set $E$ with $\delta(E)>0$, such that for all $x\in E,\delta(D_x(E))\leq C_d\delta(E)$.
\end{proposition}
This example shows that in Theorem \ref{prop1}, we cannot expect a better exponent of $\delta(A)$ such as $\delta(A)^t$ for $t<1$. We make this more precise in {Example} \ref{sharpexample}.

In the study of Falconer's distance set problem, as they are applied to the general distance set, Fourier-analytic techniques work equally well and often imply that large sets yield large pinned distance sets (see \cite{LiuFormular,GIOW5/4}).   Therefore, one may assume that there is a certain ``equivalency" between the general distance set and the pinned distance set in such type of problems. However, our construction shows that this implication can fail in the context of Bourgain-type density results. This highlights a subtle geometric divergence between pinned and unpinned frameworks as well as that between Hausdorff dimension and upper density. 

The methods used to address these two types of problems are quite different. For example, in the planar Falconer's problem, the multiscale analysis tools such as induction-on-scales, wave packet decomposition, and decoupling are heavily used \cite{GIOW5/4}. While in the upper density setting, multilinear analysis tools such as multilinear singular integrals and Gowers' norm are applied, see \cite{Bourgain1986,CMP,lyallmagyar}. It is interesting to interpret these differences and make the heuristics more precise.

The method we used in the case of two point pattern (distance) can be generalized to study some more complicated patterns in all dimensions by combining other ingredients from combinatorics.

The structure of the note is as follows: In Section \ref{sec2conterexample}, we introduce the counterexample which is also the proof of Proposition \ref{counterexample}. In Section \ref{sec3PositiveResult}, we present the proof of the positive result, Theorem \ref{prop1}. In Section \ref{sec4Sharpness}, we prove that the result of Theorem \ref{prop1} is sharp, which is Proposition \ref{sharpexample}.
\section{{Proof of Proposition \ref{counterexample}: A counterexample}}\label{sec2conterexample}
\begin{example}[rapidly inflating boxes]\label{NoBourgainTypeResult}
     For $d\geq 2$ and $\eps=\frac{1}{10^5d}$, we define the following rapidly increasing sequences. For $i=1,2,...$,
     \begin{align*}
 R_1=1, ~R_{i+1}&>\frac{100d}{\eps}\sum_{1\leq j\leq i}R_j,\\
   \ell_0=1,~ \ell_i&=\eps R_{i+1},\\
        x_i^1&=\sum_{1\leq j\leq i}(\ell_{j-1}+R_j).
     \end{align*}
    We also define set $E\subseteq  [0,\infty)^d$ as 
     \begin{equation*}
    E=\bigsqcup_{i=1}^\infty Q_i:=\bigsqcup_{i=1}^\infty\left((x_i^1,0,...,0)+Q_i'\right)
     \end{equation*}
     where cube $Q_i'=\{(x^1,...,x^d)\in\R^d: 0\leq x^j\leq \ell_i, 1\leq j\leq d\}$. The picture of the planar case is shown in Figure \ref{E in Planar case}. Roughly speaking, it is a series of inflating boxes. Each fixed box $Q_i$ has the following properties:
     \begin{itemize}
         \item Its side length $\ell_i$ is far greater than its distance to the origin, which is $x^1_i$. 
         
         This is because, by the definitions of the sequences, 
         \begin{equation*}
             x_i^1=\sum_{1\leq j\leq i}(\ell_{j-1}+R_j)=2+(1+\eps)\sum_{2\leq j\leq i}R_j<\frac{\eps}{50d}\frac{100d}{\eps}\sum_{1\leq j\leq i}R_j<\frac{\eps}{50d}R_{i+1}=\frac{\ell_i}{50d}.
         \end{equation*}
         \item The area of the $Q_i$, which is $\ell_i^d$, accounts for most of the area of $(0,x_i^1+\ell_i)^d$. 

         This is a corollary of the first item. By definition, the proportion of the areas is
         \begin{equation*}
             \frac{\ell_i^d}{(x_i^1+\ell_i)^d}=(\frac{\ell_i}{x_i^1+\ell_i})^d=\Big(\frac{1}{\frac{x_i^1}{\ell_i}+1}\Big)^d\approx 1\text{, since $\ell_i\gg x_i^1$}.
         \end{equation*}
     \end{itemize}
     \begin{figure}[h]
         \centering

\tikzset{every picture/.style={line width=0.75pt}} 

\begin{tikzpicture}[x=0.75pt,y=0.75pt,yscale=-1,xscale=1]

\draw    (82,210.25) -- (594,212.24) ;
\draw [shift={(596,212.25)}, rotate = 180.22] [color={rgb, 255:red, 0; green, 0; blue, 0 }  ][line width=0.75]    (10.93,-3.29) .. controls (6.95,-1.4) and (3.31,-0.3) .. (0,0) .. controls (3.31,0.3) and (6.95,1.4) .. (10.93,3.29)   ;
\draw    (82,210.25) -- (82,106.25) ;
\draw [shift={(82,104.25)}, rotate = 90] [color={rgb, 255:red, 0; green, 0; blue, 0 }  ][line width=0.75]    (10.93,-3.29) .. controls (6.95,-1.4) and (3.31,-0.3) .. (0,0) .. controls (3.31,0.3) and (6.95,1.4) .. (10.93,3.29)   ;
\draw  [color={rgb, 255:red, 208; green, 2; blue, 27 }  ,draw opacity=1 ] (97.88,174.17) -- (133.8,174.3) -- (133.68,210.23) -- (97.75,210.1) -- cycle ;
\draw  [color={rgb, 255:red, 208; green, 2; blue, 27 }  ,draw opacity=1 ] (233.97,132.95) -- (312.12,133.23) -- (311.85,211.38) -- (233.69,211.11) -- cycle ;
\draw  [color={rgb, 255:red, 208; green, 2; blue, 27 }  ,draw opacity=1 ] (491.31,106.83) -- (596.37,107.19) -- (596,212.25) -- (490.94,211.88) -- cycle ;
\draw    (90,204.25) .. controls (78.24,201.31) and (88.57,199.33) .. (62.63,197.37) ;
\draw [shift={(61,197.25)}, rotate = 4.09] [color={rgb, 255:red, 0; green, 0; blue, 0 }  ][line width=0.75]    (10.93,-3.29) .. controls (6.95,-1.4) and (3.31,-0.3) .. (0,0) .. controls (3.31,0.3) and (6.95,1.4) .. (10.93,3.29)   ;

\draw (41,196.65) node [anchor=north west][inner sep=0.75pt]    {$R_{1}$};
\draw (176,214.4) node [anchor=north west][inner sep=0.75pt]    {$R_{2}$};
\draw (382,218.4) node [anchor=north west][inner sep=0.75pt]    {$R_{3}$};
\draw (110,150.4) node [anchor=north west][inner sep=0.75pt]    {$\ell _{1}$};
\draw (264,111.4) node [anchor=north west][inner sep=0.75pt]    {$\ell _{2}$};
\draw (526,87.4) node [anchor=north west][inner sep=0.75pt]    {$\ell _{3}$};
\draw (91.75,215.5) node [anchor=north west][inner sep=0.75pt]    {$x_{1}^{1}$};
\draw (226,214.4) node [anchor=north west][inner sep=0.75pt]    {$x_{2}^{1}$};
\draw (484,222.4) node [anchor=north west][inner sep=0.75pt]    {$x_{3}^{1}$};
\draw (327,143.4) node [anchor=north west][inner sep=0.75pt]    {$Q_{2}$};
\draw (453,106.4) node [anchor=north west][inner sep=0.75pt]    {$Q_{3}$};
\draw (147,167.4) node [anchor=north west][inner sep=0.75pt]    {$Q_{1}$};

\end{tikzpicture}

         \caption{$E$ in the planar case}
         \label{E in Planar case}
     \end{figure}
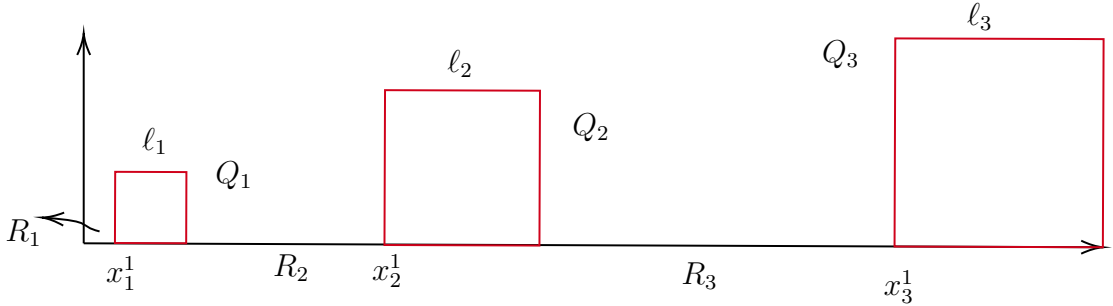  
With this in mind, readers can convince themselves that 
    \begin{equation}\label{upperdensityforEx1}
       \delta(E)= \limsup_{R\to \infty}\frac{|E\bigcap B(0,R)|}{R^d}\approx_d 1.
    \end{equation}
   If $R$ goes to infinity along each lower-right vertex, the second item tells us that $Q_i$ is almost the entire $(0,x_i^1+\ell_i)^d$, whereas $(0,x_i^1+\ell_i)^d$ is roughly the $\frac{1}{2^d}$ part of our test ball $B(0,x_i^1+\ell_i)$.
   
   Taking equation \eqref{upperdensityforEx1} for granted for now, we claim 
    \begin{lemma}
         Fix any point $x\in E$ and all $\ell>0$, there exists $\ell'>\ell$ such that for all $y\in E, |x-y|\neq \ell'.$ 
    \end{lemma}
      \begin{proof}
       Indeed, it suffices to prove the pinned distance set
        \begin{equation*}
           D_x(E):= \{|x-y|:y\in E\}\subseteq \bigsqcup[a_i,b_i],
        \end{equation*}
        where $\lim_{i\to\infty} (a_i-b_{i-1})>0$ and $\lim a_i=\lim b_i=\infty$ (in fact, in this example, we can even prove that $\lim_{i\to\infty} (a_i-b_{i-1})=\infty$). If this is the case, for any $\ell>0$, choose a sufficiently large $\ell<a_I<b_I\ll a_{I+1}<b_{I+1}$. Then points between $b_I$ and $a_{I+1}$ can be chosen as $\ell'$.
        
        Assume $x\in Q_{i_0}$. 
        It suffices to consider sufficiently large distances {(a similar rigorous argument can be found in equation \eqref{2ineq})}. 
       Choose $M\gg i_0$ such that the following holds. 
       For $m\geq M$ and $y\in Q_m$, one has  
        \begin{equation}\label{est1}
           R_{m}<x^1_m-(x^1_{i_0}+\ell_{i_0}) \leq |x-y|\leq \sqrt{(x^1_m+\ell_m-x^1_{i_0})^2+(d-1)\ell_m^2}<2\eps \sqrt d R_{m+1}.
        \end{equation}
        In the last inequality we use the rapidly increasing property of $R_i$ and $m\gg i_0$. Intuitively, owing to our choice of rapidly increasing $R_i,$ the side length of $Q_m$, which is $\ell_m$, will dominate the sum of squares. Other inequalities come from basic geometric facts whose planar case is presented in Figure \ref{Geometric facts in equation 1}.
    \begin{figure}[h]
        \centering

\tikzset{every picture/.style={line width=0.75pt}} 

\begin{tikzpicture}[x=0.75pt,y=0.75pt,yscale=-1,xscale=1]

\draw    (86,230.25) -- (598,232.24) ;
\draw [shift={(600,232.25)}, rotate = 180.22] [color={rgb, 255:red, 0; green, 0; blue, 0 }  ][line width=0.75]    (10.93,-3.29) .. controls (6.95,-1.4) and (3.31,-0.3) .. (0,0) .. controls (3.31,0.3) and (6.95,1.4) .. (10.93,3.29)   ;
\draw    (86,230.25) -- (86,126.25) ;
\draw [shift={(86,124.25)}, rotate = 90] [color={rgb, 255:red, 0; green, 0; blue, 0 }  ][line width=0.75]    (10.93,-3.29) .. controls (6.95,-1.4) and (3.31,-0.3) .. (0,0) .. controls (3.31,0.3) and (6.95,1.4) .. (10.93,3.29)   ;
\draw  [color={rgb, 255:red, 208; green, 2; blue, 27 }  ,draw opacity=1 ] (149.88,194.17) -- (185.8,194.3) -- (185.68,230.23) -- (149.75,230.1) -- cycle ;
\draw  [color={rgb, 255:red, 208; green, 2; blue, 27 }  ,draw opacity=1 ] (495.31,126.83) -- (600.37,127.19) -- (600,232.25) -- (494.94,231.88) -- cycle ;
\draw  [dash pattern={on 4.5pt off 4.5pt}]  (176,206.06) -- (552,144.06) ;
\draw [shift={(552,144.06)}, rotate = 350.64] [color={rgb, 255:red, 0; green, 0; blue, 0 }  ][fill={rgb, 255:red, 0; green, 0; blue, 0 }  ][line width=0.75]      (0, 0) circle [x radius= 3.35, y radius= 3.35]   ;
\draw [shift={(176,206.06)}, rotate = 350.64] [color={rgb, 255:red, 0; green, 0; blue, 0 }  ][fill={rgb, 255:red, 0; green, 0; blue, 0 }  ][line width=0.75]      (0, 0) circle [x radius= 3.35, y radius= 3.35]   ;
\draw   (188,236.06) .. controls (187.99,240.73) and (190.31,243.07) .. (194.98,243.09) -- (330.48,243.54) .. controls (337.15,243.56) and (340.47,245.9) .. (340.45,250.57) .. controls (340.47,245.9) and (343.81,243.58) .. (350.48,243.6)(347.48,243.59) -- (485.98,244.05) .. controls (490.65,244.06) and (492.99,241.74) .. (493,237.07) ;
\draw  [color={rgb, 255:red, 208; green, 2; blue, 27 }  ,draw opacity=1 ] (330.88,174.17) -- (388.13,174.37) -- (387.93,231.63) -- (330.68,231.43) -- cycle ;
\draw   (490,220.06) .. controls (490,215.39) and (487.67,213.06) .. (483,213.06) -- (450,213.06) .. controls (443.33,213.06) and (440,210.73) .. (440,206.06) .. controls (440,210.73) and (436.67,213.06) .. (430,213.06)(433,213.06) -- (397,213.06) .. controls (392.33,213.06) and (390,215.39) .. (390,220.06) ;
\draw  [dash pattern={on 0.84pt off 2.51pt}]  (176,206.06) -- (176,231.06) ;
\draw  [dash pattern={on 0.84pt off 2.51pt}]  (552,144.06) -- (552,231.06) ;
\draw [line width=1.5]    (149.75,230.1) -- (600.37,127.19) ;
\draw    (258,206.06) .. controls (230.28,188.24) and (201.58,200.81) .. (209.74,154.48) ;
\draw [shift={(210,153.06)}, rotate = 100.62] [color={rgb, 255:red, 0; green, 0; blue, 0 }  ][line width=0.75]    (10.93,-3.29) .. controls (6.95,-1.4) and (3.31,-0.3) .. (0,0) .. controls (3.31,0.3) and (6.95,1.4) .. (10.93,3.29)   ;

\draw (166,237.4) node [anchor=north west][inner sep=0.75pt]    {$\ell _{i_{0}}$};
\draw (551,240.4) node [anchor=north west][inner sep=0.75pt]    {$\ell _{m}$};
\draw (130.75,234.5) node [anchor=north west][inner sep=0.75pt]    {$x_{i_{0}}^{1}$};
\draw (490,239.4) node [anchor=north west][inner sep=0.75pt]    {$x_{m}^{1}$};
\draw (453,117.65) node [anchor=north west][inner sep=0.75pt]    {$Q_{m}$};
\draw (115.68,186.63) node [anchor=north west][inner sep=0.75pt]    {$Q_{i_{0}}$};
\draw (167,177.4) node [anchor=north west][inner sep=0.75pt]    {$x$};
\draw (554,147.46) node [anchor=north west][inner sep=0.75pt]    {$y$};
\draw (263,165.4) node [anchor=north west][inner sep=0.75pt]  [color={rgb, 255:red, 74; green, 144; blue, 226 }  ,opacity=1 ]  {$|x-y|$};
\draw (285,266.4) node [anchor=north west][inner sep=0.75pt]  [color={rgb, 255:red, 74; green, 144; blue, 226 }  ,opacity=1 ]  {$x_{m}^{1} -\left( x_{i_{0}}^{1} +\ell _{i_0}\right)$};
\draw (344,191.65) node [anchor=north west][inner sep=0.75pt]    {$Q_{m-1}$};
\draw (424,185.4) node [anchor=north west][inner sep=0.75pt]  [color={rgb, 255:red, 74; green, 144; blue, 226 }  ,opacity=1 ]  {$R_{m}$};
\draw (146,114.4) node [anchor=north west][inner sep=0.75pt]  [color={rgb, 255:red, 74; green, 144; blue, 226 }  ,opacity=1 ]  {$\sqrt{\left( x_{m}^{1} +\ell _{m} -x_{i_{0}}^{1}\right)^{2} +( d-1) \ell _{m}^{2}} \ $};

\end{tikzpicture}

        \caption{Geometric facts in equation \eqref{est1}}
        \label{Geometric facts in equation 1}
    \end{figure}
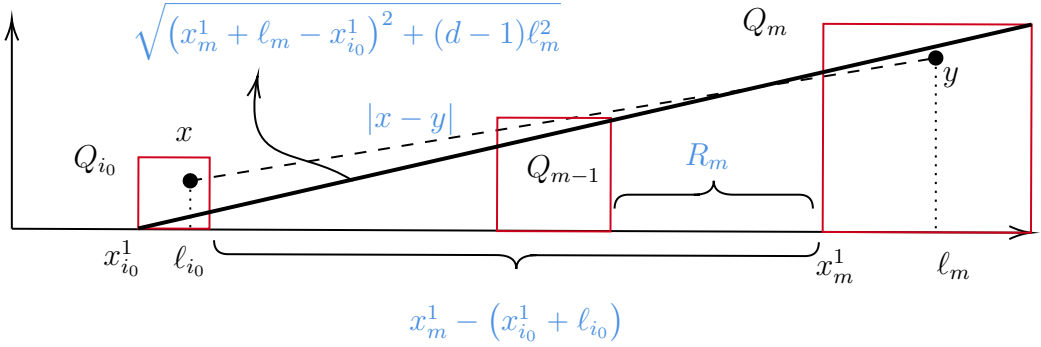
        
        Therefore, the distances between $x$ and points $y\in Q_m$ are contained in $[R_{m},2\eps \sqrt d R_{m+1}]$ when $m\geq M$. 
      Since $\eps=\frac{1}{10^5d}$, $2\eps \sqrt dR_{m+1}=\frac{2}{10^5\sqrt d}R_{m+1}\ll R_{m+1}$ and the gap tends to infinity as $m\to\infty$. Obviously, both $R_{m}$ and $2\eps \sqrt dR_{m+1}$ go to infinity. Therefore, the desired result follows. 
        \end{proof}
        
\end{example}
We conclude this section by proving equation \eqref{upperdensityforEx1}.
\begin{proof}[Proof of equation \eqref{upperdensityforEx1}]
    By the definition of limit superior,
    \begin{equation}\label{lowerbound}
        \delta(E)\geq \lim_{i\to \infty}\frac{|E\bigcap B(0,2\eps d R_i)|}{(2\eps d R_i)^d}.
    \end{equation}
    To estimate the numerator of the right hand side, we claim
    \begin{equation}\label{daxiaoguanxi}
        x^1_{i-1}+d\ell_{i-1}< 2\eps d R_i<x_{i}^1.
    \end{equation}
    Indeed, by our definitions for the sequences, when $\eps\ll 1,$
    \begin{align*}
           x_{i-1}^1+d\ell_{i-1}=2+(1+\eps)\sum_{2\leq j\leq i-1}R_j+\eps dR_i&<\frac{\eps}{50d}\frac{100d}{\eps}\sum_{1\leq j\leq  i-1}R_j+\eps dR_i<\frac{\eps}{50d}R_i+\eps dR_i<2\eps dR_i,\\
           x_i^1=2+(1+\eps)\sum_{2\leq j\leq i}R_j&>R_i>2\eps dR_i.
    \end{align*}
    Comparing the quantities with the dimensions of boxes and gaps, we find equation \eqref{daxiaoguanxi} implies 
    \begin{equation*}
        E\bigcap B(0,2\eps dR_i)=\bigsqcup_{j=1}^{i-1}Q_j.
    \end{equation*}
    Indeed, for all $x\in Q_{i-1},$ $|x|\leq  x^1_{i-1}+d\ell_{i-1}$. Therefore, $Q_{i-1}$ and all boxes smaller than it, are contained in the ball $B(0,2\eps d R_i).$ For all $y\in B(0,2\eps dR_i),$ $|y|\leq 2\eps dR_i<x_i^1=\text{dist}(Q_i,0)$ which implies $B(0,2\eps dR_i)\cap Q_i=\emptyset.$
    
    Consequently,  by the definition of $E$,
     \begin{equation*}
      \text{RHS of \eqref{lowerbound}}=  \lim_{i\to \infty}\frac{|E\bigcap B(0,2\eps dR_i)|}{(2\eps dR_i)^d}= \lim_{i\to\infty} \frac{\sum_{j<i} |Q_j|}{(2\eps dR_i)^d}.
    \end{equation*}   
Then by the definition of $Q_{i-1}$, we have
    \begin{equation*}
        \lim_{i\to\infty} \frac{\sum_{j<i} |Q_j|}{(2\eps dR_i)^d}>\lim_{i\to\infty} \frac{|Q_{i-1}|}{(2\eps dR_i)^d}=\lim_{i\to\infty} \frac{(\eps R_i)^d}{(2\eps dR_i)^d}=(2d)^{-d}.
    \end{equation*}
    This means $\delta(E)\geq C_d$, while trivially, by its definition, $\delta(E)\lesssim_d  1$, which implies equation \eqref{upperdensityforEx1}.
\end{proof}

\section{Positive results for 2-point pattern}\label{sec3PositiveResult}
Although the above example shows that the Bourgain type result  cannot hold, one has the weaker positive conclusion, Theorem \ref{prop1}.
The theorem is derived from the following two lemmas. Lemma \ref{translationinvariance} is the reason why we can obtain ``uniformity" over all $x$ and its proof depends heavily on the definition of upper density. Lemma \ref{0caseforprop1} is the ``pinned at the origin" case which is equivalent to Theorem \ref{prop1} due to the uniformity.
\begin{lemma}\label{0caseforprop1}
   Assume $d\geq 2$ and $A\subseteq \R^d$ with $\delta(A)>0$. Then $\delta(D_0(A))\geq \frac{\delta(A)}{2|\mathbb S^{d-1}|}$.
\end{lemma}
\begin{lemma}\label{translationinvariance}
    For all $A\subseteq \R^d$ and $x\in \R^d$, $\delta(A-x)=\delta(A).$
\end{lemma}
\begin{proof}[Proof of Theorem \ref{prop1} assuming Lemma \ref{0caseforprop1} and \ref{translationinvariance}] For all $x\in \R^d$, by Lemma \ref{translationinvariance}, $\delta(A-x)=\delta(A)$. 
Then by applying Lemma \ref{0caseforprop1} to $(A-x)$ and translation invariance, $\delta(D_x(A))=\delta(D_0(A-x))\geq \frac{\delta(A-x)}{2|\mathbb S^{d-1}|}=\frac{\delta(A)}{2|\mathbb S^{d-1}|}.$
\end{proof}
We first prove Lemma \ref{translationinvariance}.
\begin{proof}[Proof of Lemma \ref{translationinvariance}]By symmetry, it suffices to prove for all $A\subseteq \R^d$ and any fixed $x\in \R^d,$
\begin{equation}\label{onesideestimate}
    \limsup_{R\to\infty}\frac{|(A-x)\bigcap B(0,R)|}{R^d}\leq \limsup_{R\to\infty}\frac{|A\bigcap B(0,R)|}{R^d}.
\end{equation}
Indeed, applying $x=-x$ in equation \eqref{onesideestimate}, we obtain
\begin{equation*}
    \limsup_{R\to\infty}\frac{|(A+x)\bigcap B(0,R)|}{R^d}\leq \limsup_{R\to\infty}\frac{|A\bigcap B(0,R)|}{R^d}.
\end{equation*}
Replace $A$ by $(A-x)$, and note that $(A-x)+x=A$,  
\begin{equation*}
    \limsup_{R\to\infty}\frac{|A\bigcap B(0,R)|}{R^d}\leq \limsup_{R\to\infty}\frac{|(A-x)\bigcap B(0,R)|}{R^d}.
\end{equation*}
Combining this with equation \eqref{onesideestimate}, we obtain the desired equality. To prove equation \eqref{onesideestimate}, by translation invariance and additivity of Lebesgue measure,
\begin{align*}
    \frac{|A\bigcap B(0,R)|}{R^d}&= \frac{|(A-x)\bigcap B(-x,R)|}{R^d}\\
    &=\frac{|(A-x)\bigcap B(-x,R)\bigcap B(0,R)|}{R^d}+\frac{|(A-x)\bigcap B(-x,R)\setminus B(0,R)|}{R^d}\\
    &\leq \frac{|(A-x)\bigcap B(0,R)|}{R^d}+ \frac{|B(-x,R)\setminus B(0,R) |}{R^d}\\
    &\to \frac{|(A-x)\bigcap B(0,R)|}{R^d}, R\to \infty. 
\end{align*} 
    In the limit argument of last line, we use the fact that the second term $\frac{|B(-x,R)\setminus B(0,R) |}{R^d}$ of the last inequality tends to $0$ as $R\to\infty$. In fact, by a basic geometric property of balls (see Figure \ref{xianrannnnn}), it can be easily checked that $B(-x,R)\setminus B(0,R) \subseteq B(0,R+|x|)\setminus B(0,R)$ while $|B(0,R+|x|)\setminus B(0,R)|=C_d|x|R^{d-1}$ for some constant $C_d$. So 
    \begin{equation*}
      \frac{|B(-x,R)\setminus B(0,R)|}{R^d}\leq \frac{C_d|x|R^{d-1}}{R^d}=\frac{C_d|x|}{R}\to 0. \qedhere
    \end{equation*}
    \begin{figure}[h]
        \centering
        
\tikzset{every picture/.style={line width=0.75pt}} 

\begin{tikzpicture}[x=0.75pt,y=0.75pt,yscale=-1,xscale=1]

\draw [color={rgb, 255:red, 208; green, 2; blue, 27 }  ,draw opacity=1 ]   (176.15,203.85) ;
\draw [shift={(176.15,203.85)}, rotate = 0] [color={rgb, 255:red, 208; green, 2; blue, 27 }  ,draw opacity=1 ][fill={rgb, 255:red, 208; green, 2; blue, 27 }  ,fill opacity=1 ][line width=0.75]      (0, 0) circle [x radius= 2.34, y radius= 2.34]   ;
\draw  [line width=1.5]  (98.31,203.85) .. controls (98.31,147.05) and (144.35,101) .. (201.15,101) .. controls (257.95,101) and (304,147.05) .. (304,203.85) .. controls (304,260.65) and (257.95,306.69) .. (201.15,306.69) .. controls (144.35,306.69) and (98.31,260.65) .. (98.31,203.85) -- cycle ;
\draw    (201.15,203.85) ;
\draw [shift={(201.15,203.85)}, rotate = 0] [color={rgb, 255:red, 0; green, 0; blue, 0 }  ][fill={rgb, 255:red, 0; green, 0; blue, 0 }  ][line width=0.75]      (0, 0) circle [x radius= 2.34, y radius= 2.34]   ;
\draw [line width=1.5]  [dash pattern={on 1.69pt off 2.76pt}]  (177.2,203.85) -- (304,203.85) ;
\draw   (275.38,197.3) .. controls (275.38,192.63) and (273.05,190.3) .. (268.38,190.3) -- (236.88,190.3) .. controls (230.21,190.3) and (226.88,187.97) .. (226.88,183.3) .. controls (226.88,187.97) and (223.55,190.3) .. (216.88,190.3)(219.88,190.3) -- (185.38,190.3) .. controls (180.71,190.3) and (178.38,192.63) .. (178.38,197.3) ;
\draw    (259.82,120.14) -- (271.82,163.14) ;
\draw    (279,203.85) -- (291,246.85) ;
\draw    (279.82,140.14) -- (291.82,183.14) ;
\draw    (271.82,242.14) -- (278.82,267.14) ;
\draw    (259.82,261.14) -- (265.82,280.14) ;
\draw    (237.82,110.14) -- (241.82,123.14) ;
\draw    (295.82,207.14) -- (300.82,225.14) ;
\draw    (242.82,283.14) -- (246.82,294.14) ;
\draw    (294,225) .. controls (322.04,219.08) and (321.49,212.2) .. (347.27,245.44) ;
\draw [shift={(348.47,246.99)}, rotate = 232.35] [color={rgb, 255:red, 0; green, 0; blue, 0 }  ][line width=0.75]    (10.93,-3.29) .. controls (6.95,-1.4) and (3.31,-0.3) .. (0,0) .. controls (3.31,0.3) and (6.95,1.4) .. (10.93,3.29)   ;
\draw  [color={rgb, 255:red, 208; green, 2; blue, 27 }  ,draw opacity=1 ][fill={rgb, 255:red, 255; green, 255; blue, 255 }  ,fill opacity=0 ,even odd rule][line width=1.5]  (74.67,204.72) .. controls (74.18,147.7) and (119.69,101.08) .. (176.32,100.59) .. controls (232.94,100.11) and (279.24,145.94) .. (279.73,202.97) .. controls (280.22,260) and (234.71,306.62) .. (178.08,307.1) .. controls (121.46,307.59) and (75.16,261.75) .. (74.67,204.72)(48.45,204.95) .. controls (47.84,133.44) and (104.99,74.98) .. (176.09,74.37) .. controls (247.2,73.77) and (305.34,131.24) .. (305.95,202.75) .. controls (306.56,274.25) and (249.41,332.71) .. (178.31,333.32) .. controls (107.2,333.93) and (49.07,276.45) .. (48.45,204.95) ;
\draw    (231,95) .. controls (280.71,94) and (315.42,94.96) .. (344.16,115.06) ;
\draw [shift={(345.47,115.99)}, rotate = 215.91] [color={rgb, 255:red, 0; green, 0; blue, 0 }  ][line width=0.75]    (10.93,-3.29) .. controls (6.95,-1.4) and (3.31,-0.3) .. (0,0) .. controls (3.31,0.3) and (6.95,1.4) .. (10.93,3.29)   ;

\draw (160,206.4) node [anchor=north west][inner sep=0.75pt]    {$O$};
\draw (203.15,207.25) node [anchor=north west][inner sep=0.75pt]    {$-x$};
\draw (282,185.4) node [anchor=north west][inner sep=0.75pt]    {$|x|$};
\draw (222,164.4) node [anchor=north west][inner sep=0.75pt]    {$R$};
\draw (13,128.73) node [anchor=north west][inner sep=0.75pt]  [color={rgb, 255:red, 0; green, 0; blue, 0 }  ,opacity=1 ]  {$B( 0,|x|+R)$};
\draw (315,255.4) node [anchor=north west][inner sep=0.75pt]    {$B( -x,R) \setminus B( 0,R)$};
\draw (301,123.4) node [anchor=north west][inner sep=0.75pt]  [font=\footnotesize]  {$B( 0,|x|+R) \setminus B( 0,R)$};
\draw (46.33,188.07) node [anchor=north west][inner sep=0.75pt]  [color={rgb, 255:red, 0; green, 0; blue, 0 }  ,opacity=1 ]  {$B( 0, R)$};
\draw (91.67,242.07) node [anchor=north west][inner sep=0.75pt]    {$B( -x,R)$};

\end{tikzpicture}

        \caption{Basic geometric fact}
        \label{xianrannnnn}
    \end{figure}
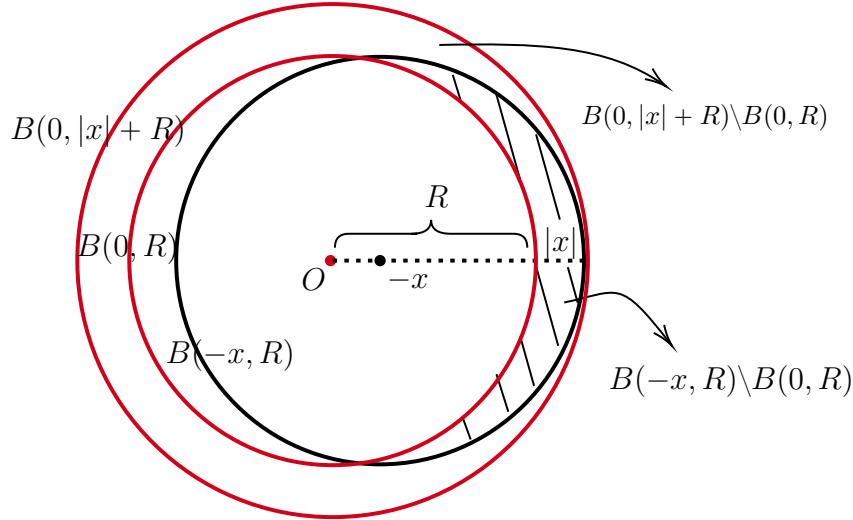
\end{proof}
\begin{remark}\label{basicproperty}
    By a similar argument involving only definition, we can show that the upper density has sub-additivity and monotonicity. For any $A,B$
    \begin{equation}\label{submon}
        \delta(A\cup B)\leq \delta(A)+\delta(B),\quad \delta(A)\leq \delta(A\cup B).
    \end{equation}
\end{remark}
We now start to prove Lemma \ref{0caseforprop1} via a spherical integration argument.
\begin{proof}[Proof of Lemma \ref{0caseforprop1}]
   By the definition of limit superior, there is $R_i\to\infty,i\to \infty$ such that 
    \begin{equation}\label{upseq}
        \left|A\bigcap B(0,R_i)\right|\geq \frac{\delta(A)}{2}R_i^d.
    \end{equation}
    When $i$ is sufficiently large, by spherical coordinates, equation \eqref{upseq} implies that  
    \begin{align*}
        \frac{\delta(A)}{2}R_i^d&\leq \int_{B(0,R_i)}\chi_A(y)dy\\
       &=\int_{0}^{R_i}\int_{\mathbb S^{d-1}}\chi_A(ry')r^{d-1}d\sigma(y')dr,
    \end{align*}
    where $\chi_A$ is the characteristic function of set $A$ and $d\sigma$ is the area measure on unit sphere $\mathbb S^{d-1}$. Let $A_0^i$ denote all possible distances smaller than $R_i$, between $0$ and  points in $A$, i.e.
    \begin{equation}\label{Def:A_o^i}
         \begin{aligned}
        A^i_0:=&\{r\in [0,R_i]: r\mathbb S^{d-1}\bigcap A\neq \emptyset\}\\
        =&\{r\in [0,R_i]: \exists y\in A \text{ such that } |y-0|=|y|=r\}\\
        =&D_0(A)\bigcap [0,R_i].
    \end{aligned}
    \end{equation}
    Then   \begin{equation}\label{proofxs}
           \begin{aligned}
         \frac{\delta(A)}{2}R_i^d&\leq  \int_{0}^{R_i}\int_{\mathbb S^{d-1}}\chi_A(ry')r^{d-1}d\sigma(y')dr\\
        &=\int_{A^i_0}\int_{\mathbb S^{d-1}}\chi_A(ry')r^{d-1}d\sigma(y')dr\\
        &\leq |\mathbb S^{d-1}|\int_{A^i_0} r^{d-1}dr\leq |\mathbb S^{d-1}||A^i_0| R_i^{d-1},
    \end{aligned}
    \end{equation}
  where $|A^i_0|$ is the one-dimensional Lebesgue measure of $A^i_0$. This implies 
    \begin{equation*}
        \frac{\delta(A)}{2}R_i\leq |\mathbb S^{d-1}||A^i_0|.
    \end{equation*}
 Recalling the definition of $A_0^i$ in \eqref{Def:A_o^i}, we obtain that
 \begin{equation*}
     \delta(D_0(A))\geq \lim_{i\to \infty} \frac{|D_0(A)\bigcap[0,R_i]|}{R_i}=\lim_{i\to\infty}\frac{|A_0^i|}{R_i}\geq \frac{\delta(A)}{2|\mathbb S^{d-1}|}. \qedhere
 \end{equation*}
\end{proof}

\section{Sharpness of Theorem \ref{prop1}} \label{sec4Sharpness}
Theorem \ref{prop1} is sharp up to a constant depending on $d$. Recall Proposition \ref{Prop:Sharpness}.
Its proof is the following Example \ref{sharpexample}. In what follows, $C_d$ may change from line to line.
\begin{example}[thin annulus]\label{sharpexample}
    Fix $\eps_0\ll 1$ and let $R_i$ be the rapidly increasing sequence in Example \ref{NoBourgainTypeResult} with $\eps$ replaced with $\eps_0$. Define 
    \begin{equation*}
        I:=\bigsqcup_{i\geq 100} [R_i,(1+\eps_0)R_i]=:\bigsqcup_{i\geq 100} I_i,~E:=\bigsqcup_{i\geq100}\{r\mathbb S^{d-1}:r\in I_i\}=:\bigsqcup_{i\geq100}S_i.
    \end{equation*}
    Each $S_i$ is an annulus with volume $|S_i|=c_dR_i^d[(1+\eps_0)^d-1]$. Similar to what we discussed in Example \ref{NoBourgainTypeResult}, due to the rapidly increasing property of $R_i$, the volume of the largest annulus will dominate, i.e. $|\bigsqcup_{j<i} S_j|\leq \frac{1}{10}|S_i|$. Indeed, 
    \begin{equation*}
        \big|\bigsqcup_{j<i} S_j\big|=c_d[(1+\eps_0)^d-1]\sum_{j<i}R_j^d\leq \frac{1}{10}c_d[(1+\eps_0)^d-1]\left(\frac{100d}{\eps_0}\sum_{j<i}R_j\right)^d<\frac{1}{10}|S_i|.
    \end{equation*}
     In the second inequality, we use the fact that $\frac{1}{10}\times \left(\frac{100d}{\eps_0}\right)^d\geq 1$, as long as $\eps_0$ is sufficiently small. Therefore, it can be deduced that 
    \begin{equation}\label{deltaE}
        \delta(E)\geq \lim_{i\to\infty}\frac{|S_i|}{((1+\eps_0)R_i)^d}=\lim_{i\to\infty}\frac{c_dR_i^d[(1+\eps_0)^d-1]}{((1+\eps_0)R_i)^d}=C_d\eps_0.
    \end{equation}
    Now we analyze the upper density of the pinned distance set and our goal is to prove $\delta(D_x(A))=O_d(\eps_0)$. For each fixed $x\in E$, take $M$ sufficiently large, such that for all $i\geq M$ 
    \begin{equation*}
        D_x(S_i)\subseteq [R_i-|x|,(1+\eps_0)R_i+|x|]\subseteq [(1-\frac{\eps_0}{2})R_i,(1+\frac{3\eps_0}{2})R_i]=:\tilde{I_i}.
    \end{equation*}
    See Figure \ref{planarshells} for the planar case. This roughly says that the scale we are working with is too large so that we can treat $x$ as 0.
    \begin{figure}[h]
        \centering

\tikzset{every picture/.style={line width=0.75pt}} 

\begin{tikzpicture}[x=0.75pt,y=0.75pt,yscale=-1,xscale=1]

\draw   (148.67,254.1) .. controls (148.67,190.35) and (201.76,138.67) .. (267.25,138.67) .. controls (332.75,138.67) and (385.84,190.35) .. (385.84,254.1) .. controls (385.84,317.85) and (332.75,369.53) .. (267.25,369.53) .. controls (201.76,369.53) and (148.67,317.85) .. (148.67,254.1) -- cycle ;
\draw   (164.28,254.1) .. controls (164.28,198.74) and (210.38,153.86) .. (267.25,153.86) .. controls (324.13,153.86) and (370.23,198.74) .. (370.23,254.1) .. controls (370.23,309.46) and (324.13,354.34) .. (267.25,354.34) .. controls (210.38,354.34) and (164.28,309.46) .. (164.28,254.1) -- cycle ;
\draw    (277.92,246.77) ;
\draw [shift={(277.92,246.77)}, rotate = 0] [color={rgb, 255:red, 0; green, 0; blue, 0 }  ][fill={rgb, 255:red, 0; green, 0; blue, 0 }  ][line width=0.75]      (0, 0) circle [x radius= 1.34, y radius= 1.34]   ;
\draw    (267.25,254.1) ;
\draw [shift={(267.25,254.1)}, rotate = 0] [color={rgb, 255:red, 0; green, 0; blue, 0 }  ][fill={rgb, 255:red, 0; green, 0; blue, 0 }  ][line width=0.75]      (0, 0) circle [x radius= 1.34, y radius= 1.34]   ;
\draw    (172.45,323.56) -- (362.06,184.64) ;
\draw   (345.25,191.77) .. controls (342.46,188.03) and (339.19,187.56) .. (335.45,190.35) -- (315.46,205.29) .. controls (310.12,209.28) and (306.05,209.4) .. (303.26,205.67) .. controls (306.05,209.4) and (304.78,213.27) .. (299.44,217.26)(301.84,215.46) -- (279.45,232.2) .. controls (275.71,234.99) and (275.24,238.26) .. (278.03,242) ;
\draw   (181.47,329.36) .. controls (184.25,333.11) and (187.51,333.59) .. (191.26,330.81) -- (223.13,307.2) .. controls (228.49,303.23) and (232.56,303.12) .. (235.33,306.87) .. controls (232.56,303.12) and (233.85,299.26) .. (239.2,295.29)(236.79,297.07) -- (271.07,271.67) .. controls (274.82,268.89) and (275.31,265.63) .. (272.53,261.88) ;
\draw   (265.74,253.95) .. controls (262.97,250.18) and (259.71,249.68) .. (255.95,252.45) -- (229.56,271.82) .. controls (224.19,275.77) and (220.12,275.86) .. (217.35,272.09) .. controls (220.12,275.86) and (218.81,279.71) .. (213.44,283.65)(215.85,281.88) -- (187.04,303.03) .. controls (183.28,305.79) and (182.78,309.05) .. (185.54,312.81) ;

\draw (241.23,374.92) node [anchor=north west][inner sep=0.75pt]  [font=\footnotesize]  {$S_{i} ,\ i\geq M$};
\draw (255.88,261.28) node [anchor=north west][inner sep=0.75pt]  [font=\footnotesize]  {$O$};
\draw (282.14,246.28) node [anchor=north west][inner sep=0.75pt]    {$x$};
\draw (221.54,312.02) node [anchor=north west][inner sep=0.75pt]  [font=\footnotesize,color={rgb, 255:red, 74; green, 144; blue, 226 }  ,opacity=1 ]  {$( 1+\varepsilon _{0}) R_{i}$};
\draw (203.03,259.72) node [anchor=north west][inner sep=0.75pt]  [font=\footnotesize,color={rgb, 255:red, 74; green, 144; blue, 226 }  ,opacity=1 ]  {$R_{i}$};
\draw (253.7,232.16) node [anchor=north west][inner sep=0.75pt]  [font=\footnotesize,color={rgb, 255:red, 74; green, 144; blue, 226 }  ,opacity=1 ]  {$|x|$};
\draw (258.78,187.59) node [anchor=north west][inner sep=0.75pt]  [font=\footnotesize,color={rgb, 255:red, 74; green, 144; blue, 226 }  ,opacity=1 ]  {$R_{i} -|x|$};

\end{tikzpicture}

        \caption{Annulus in $R^2$}
        \label{planarshells}
    \end{figure}
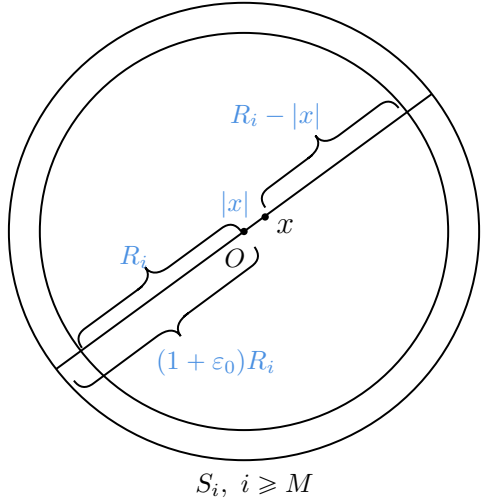  

    Therefore, we have 
    \begin{equation*}
        D_x(E)\subseteq B(0,R_M)\bigcup \bigsqcup_{i\geq M}\tilde{I_i}.
    \end{equation*}
    Similar to the analysis of the volumes of the annuli, one can show that the contribution from the largest $\tilde I_i$ dominates the total, rendering the contributions from smaller intervals negligible. 
    As a consequence, 
      \begin{equation}\label{2ineq}
        \delta(D_x(E))\leq \delta\Big(\bigsqcup_{i\geq M}\tilde{I_i}\Big)\leq C_d\eps_0.
    \end{equation}
    Combining this with \eqref{deltaE}, one can see there is a constant $C_d$ such that $\delta(D_x(E))\leq C_d\delta(E).$
    
    In the first inequality of equation \eqref{2ineq}, we use the monotonicity \eqref{submon} and the fact that $\delta(B(0,R_M))=0$ since $M$ is a fixed number. Now we prove the second inequality of it.

    Assume the sequence $\{\bar R_n\}$ satisfies  
    \begin{equation*}
        \delta\Big(\bigsqcup_{i\geq M}\tilde{I_i}\Big)=\lim_{n\to \infty}\frac{|\bigsqcup_{i\geq M}\tilde{I_i}\bigcap B(0,\bar R_n)|}{\bar R_n}.
    \end{equation*}
    For each sufficiently large $\bar R_n\geq R_{2M}$, by pigeonholing, there is $i \geq M$ such that $(1+\frac{3\eps_0}{2})R_i<\bar R_n\leq (1+\frac{3\eps_0}{2})R_{i+1}$. There are two cases for the exact location of $\bar R_n$:
    \begin{enumerate}
        \item $\bar R_n\in [(1-\frac{\eps_0}{2})R_{i+1},(1+\frac{3\eps_0}{2})R_{i+1}]=\tilde I_{i+1}$,
        \item $\bar R_n\not\in \tilde I_{i+1}$.
    \end{enumerate}
    If we are in the first case, since the longest interval dominates, this implies
    \begin{equation*}
        \frac{|\bigsqcup_{i\geq M}\tilde{I_i}\bigcap B(0,\bar R_n)|}{\bar R_n}\leq C_d \frac{|\tilde I_{i+1}|}{(1-\frac{\eps_0}{2})R_{i+1}}\leq C_d\frac{2\eps_0}{(1-\frac{\eps_0}{2})}.
    \end{equation*}
    If we are in the second case, which means that 
    \begin{equation*}
        (1+\frac{3\eps_0}{2})R_i<\bar R_n<(1-\frac{\eps_0}{2})R_{i+1},
    \end{equation*}
    then 
    \begin{equation*}
        \frac{|\bigsqcup_{i\geq M}\tilde{I_i}\bigcap B(0,\bar R_n)|}{\bar R_n}\leq C_d \frac{|\tilde I_{i}|}{(1+\frac{3\eps_0}{2})R_{i}}\leq C_d\frac{2\eps_0}{(1+\frac{3\eps_0}{2})}.
    \end{equation*}
    In both cases, the second inequality of equation \eqref{2ineq} holds if $C_d$ is appropriately chosen.
\end{example}
\begin{remark}
    Intuitively, to make the upper density of the pinned distance set as small as possible, each distance $r$ should be very ``efficient", i.e. $r\mathbb S^{d-1}$ should contain points from $E$ as many as possible. The extreme case is the entire $r\mathbb S^{d-1}\subseteq E$ which is the example. This can also be seen from the proof of Proposition \ref{prop1}: the last inequality of equation \eqref{proofxs} is an equality when $\chi_A(ry')\equiv 1$ for $y'\in \mathbb S^{d-1}$ which means $r\mathbb S^{d-1}\subseteq A$. 
\end{remark}

\bibliographystyle{plain}
\bibliography{ref.bib}

\vspace{1em}
\noindent \textsc{Department of Mathematics, 1984 Mathematics Road, The University of British Columbia Vancouver, BC Canada V6T 1Z2}. \\
\textit{Email address}: \texttt{chjwang@math.ubc.ca}

\end{document}